\documentclass[12pt]{article}

\usepackage{tikz}
\usepackage{subfigure}
\usepackage[english]{babel}

\usepackage[center]{caption2}
\usepackage{amsfonts,amssymb,amsmath,latexsym,amsthm}
\usepackage{multirow}

\def\gm{\gamma}
\def\ex{\mathrm{ex}}

\def\lf{\left\lfloor}
\def\rf{\right\rfloor}
\def\lc{\left\lceil}
\def\rc{\right\rceil}
\def\dtaT{\lf\frac{n}{2}\rf}
\def\bad{\frac13\sgmn}

\def\Fkr{F_{k,r}}
\def\geqs{\geqslant}
\def\leqs{\leqslant}

\def\sgmn{\sqrt{\gm}n}
\def\gmnn{\gm n^2}
\def\Dta{\Delta}

\def\ckq{C_{k,q}}

\newtheorem{thm}{Theorem}[section]

\newtheorem{lem}[thm]{Lemma}

\newtheorem{claim}{Claim}

\newtheorem{obs}[thm]{Observation}

\begin{document}

\title{Extremal graph for intersecting odd cycles}
\author{Xinmin Hou$^a$, \quad Yu Qiu$^b$, \quad Boyuan Liu$^c$\\
\small $^{a,b,c}$ Key Laboratory of Wu Wen-Tsun Mathematics\\
\small Chinese Academy of Sciences\\
\small School of Mathematical Sciences\\
\small University of Science and Technology of China\\
\small Hefei, Anhui 230026, China.\\
\small $^a$xmhou@ustc.edu.cn,\quad  $^b$yuqiu@mail.ustc.edu.cn\\
\small $^c$lby1055@mail.ustc.edu.cn
}

\date{}

\maketitle

\begin{abstract}
An extremal graph for a graph $H$ on $n$ vertices is a graph on $n$ vertices with maximum number of edges that does not contain $H$ as a subgraph. Let  $T_{n,r}$ be the Tur\'{a}n graph, which is the complete $r$-partite graph on $n$ vertices with part sizes that differ by at most one. The well-known Tur\'{a}n Theorem states that  $T_{n,r}$ is the only extremal graph for complete graph $K_{r+1}$. Erd\"{o}s et al. (1995) determined the extremal graphs for intersecting triangles and Chen et al. (2003) determined the maximum number of edges of the extremal graphs for intersecting cliques. In this paper, we determine the extremal graphs for intersecting odd cycles.
\end{abstract}

\section{Introduction}
In this paper, all graphs considered are simple and finite. For a graph $G$ and a vertex $x\in V(G)$, the neighborhood of $x$ in $G$ is denoted by $N_G(x)$. Let $N_G[x]=\{x\}\cup N_G(x)$.
The {\it degree} of $x$, denoted by $\deg_G(x)$, is $|N_G(x)|$.  Let  $\delta(G)$ and $\Delta(G)$ denote the minimum and maximum degrees of $G$, respectively. A {\it matching} $M$ in $G$ is a subgraph of $G$ with $\delta(M)=\Delta(M)=1$. The {\it matching number} of $G$, denoted by $\nu(G)$, is the maximum number of edges in a matching in $G$. Let $e(G)$ be the number of edges of $G$. For a graph $G$ and $S, T\subset V(G)$, let $e_G(S, T)$ be the number of edges $e=xy\in{E(G)}$ such that $x\in S$ and $y\in T$, if $S=T$, we use $e_G(S)$ instead of $e_G(S, S)$; and  we use $e_G(u, T)$ instead of $e_G(\{u\}, T)$ for convenience, the index $G$ will be omitted if no confusion from the context. For a subset $X\subseteq V(G)$ or $X\subseteq  E(G)$, let $G[S]$ be the subgraph of $G$ induced by $X$, that is $G[X]=(X, E(X))$ if $X\subseteq V(G)$, or $G[X]=(V(X),X)$ if $X\subseteq  E(G)$.  A cycle of length $q$ is called a {\it $q$-cycle}. 

Given two graphs $G$ and $H$, we say that $G$ is {\it $H$-free} if $G$ does not contain an $H$ as a subgraph. The Tur\'an function, denoted by $\ex(n,H)$, is the largest possible number of edges of an $H$-free graph on $n$ vertices. That is,
$$\ex(n,H)=\max\{e(G):|V(G)|=n,  \mbox{ $G$ is $H$-free}\}.$$
And for positive integers $n\ge r$, the Tur\'an graph, denoted by $T_{n,r}$, is the complete $r$-partite graph on $n$ vertices with part sizes that differ by at most one (also called the complete balanced $r$-partite graph). The well-known Tur\'{a}n Theorem states that $ex(n, K_{r+1})=e(T_{n,r})$ and $T_{n,r}$ is the only extremal graph for complete graph $K_{r+1}$.

A {\it $k$-fan}, denoted by $F_k$, is a graph on $2k+1$ vertices consisting of $k$ triangles which intersect in exactly one common vertex. In 1995, Erd\"{o}s et al.~\cite{Erdos} gave  the value of $ex(n, F_k)$ and determined the extremal graphs for $F_k$ as follows.
\begin{thm}[\cite{Erdos}]\label{THM:Fk}
For $k\geqs1$ and $n\geqs50k^2$,$$\ex(n,F_k)=e(T_{n,2})+g(k),$$
 where
 $$g(k)=\left\{
\begin{aligned}
&k^2-k & &\mbox{if $k$ is odd,}\\
&k^2-\frac{3}{2}k & &\mbox{if $k$ is even.}\\
\end{aligned} \right.$$
Moreover, when $k$ is odd, the extremal graph must be a $T_{n,2}$ with two vertex disjoint copies of $K_k$ embedding in one partite set. When $k$ is even, the extremal graph must be a $T_{n,2}$ with a graph having $2k-1$ vertices, $k^2-\frac{3}{2}k$ edges with maximum degree $k-1$ embedded in one partite set.
\end{thm}
In 2003, Chen et al.~\cite{Wei} proved that $ex(n, \Fkr)=e(T_{n,2})+g(k)$, where $\Fkr$ is a graph consisting of $k$ complete graphs of order $r$ which intersect in exactly one common vertex and $g(k)$ is the same as in Theorem~\ref{THM:Fk}. Recently, Gelbov~\cite{Glebov-2011} and Liu~\cite{Liu-2013} gave the extremal graphs for blow-ups of paths~\cite{Glebov-2011}, cycles and a large class of trees~\cite{Liu-2013}.

In this paper, motivated by the results in~\cite{Wei,Erdos,Glebov-2011,Liu-2013},  we generalize Theorem \ref{THM:Fk} in another way.  For a positive integer $k$ and an odd integer $q$ with $q\ge 5$, let $\ckq$ be the  graph consisting of $k$ $q$-cycles which intersect exactly in one common vertex, called the center of it.
For $n\geqs4(k-1)^2$, let $\mathcal{F}_{n,k}$ be the family of graphs with each member is a Tur\'an graph $T_{n,2}$ with a complete bipartite graph $K_{k-1,k-1}$ embedded into one class. Our main result is as follows.

\begin{thm}\label{THM:MAIN THEOREM}
 For an integer $k\geqs2$ and an odd integer $q\geqs5$, there exists $n_1(k,q)\in\mathbb{N}$ such that for all $n\geqs n_1(k,q)$, we have $$\ex(n,\ckq)=e(T_{n,2})+(k-1)^2,$$      and the only extremal graphs for $C_{k,q}$  are members of $\mathcal{F}_{n,k}$.
\end{thm}

The remaining of the paper is arranged as follows. Section 2 gives some lemmas. Section 3 gives the proof of Theorem~\ref{THM:MAIN THEOREM}.

\section{Lemmas}

The following two lemmas are useful to estimate the number of edges of a graph with restrict degree and matching number.

\begin{lem}[ Chav\'atal~\cite{Hanson}]\label{LEMMA:f_Nu_Delta} For any graph $G$ with maximum degree $\Dta\geqs1$ and matching number $\nu\geqs1$, then $e(G)\leqs f(\nu,\Delta)=\nu\Dta+{\lf\frac{\Dta}{2}\rf}{\lf\frac{\nu}{\lc\Dta/2\rc}\rf}\le \nu(\Delta+1)$.
\end{lem}
The following stability result due to Erd\"{o}s~\cite{Erdos-66} and Simonovits~\cite{Simonovits-66} gives the rough structure of the extremal graphs for a graph $H$ with $\chi(H)=r\geqs3$ and $H\neq{K_r}$.

\begin{lem}[\cite{Erdos-66,Simonovits-66}]\label{LEMMA:partition} Let $H$ be a graph with $\chi(H)=r\geqs3$ and $H\neq{K_r}$. Then, for every $\gm>0$, there exists $\delta>0$ and $n_0=n_0(H,\gm)\in\mathbb{N}$ such that the following holds. If $G$ is an $H$-free graph on $n\geqs{n_0}$ vertices with $e(G)\geqs\ex(n,H)-\delta{n^2}$, then there exists a partition of $V(G)=V_1\dot{\cup}\cdots\dot{\cup}V_{r-1}$ such that $\sum^{r-1}_{i=1}e(V_i)<\gm n^2$.
\end{lem}

The following is a simple observation.
\begin{obs}\label{OBS: o1}
Let $G$ be a graph with no isolated vertex. If $\Delta(G)\le 2$, then $$\nu(G)\ge \frac{|V(G)|-\omega(G)}2,$$ where $\omega(G)$ is the number of components of $G$.
\end{obs}
\begin{proof}
Since $\Delta(G)\le 2$, each component of $G$ is a path or a cycle. Hence each component $C$ of $G$ has matching number at least $\frac{|V(C)|-1}2$. This implies the desired result.
\end{proof}

\begin{lem}\label{LEMMA:MAIN LEMMA} Let $G$ be a graph with no isolated vertex. If for all $x\in{V(G)}$, $\deg(x)+\nu(G-N[x])\leqs{r}$, then $e(G)\leqs r^2$. Moreover, the equality holds if and only if $G=K_{r,r}$.
\end{lem}
\begin{proof}
 Clearly, $\Dta(G)\leqs{r}$. We claim that $\nu(G)\leqs{r}$. Let $u_1v_1,\cdots,u_\ell v_\ell$ be a matching in $G$. Wlog, assume that
 \begin{eqnarray*}
& &\{u_1,v_1,\cdots,u_s,v_s,u_{s+1},\cdots,u_{s+t}\}\subseteq{N[u_1]},\\
& &\{v_{s+1},\cdots,v_{s+t},u_{s+t+1},v_{s+t+1},u_\ell,v_\ell\}\subseteq{V(G)\setminus N[u_1]}.
\end{eqnarray*}
Then $s\geqs1$, $2s+t-1\leqs\deg(u_1)$ and $\ell-(s+t)\leqs\nu(G-N[u_1])$. Thus
 $$\ell\leqs\nu(G-N[u_1])+s+t\leqs\deg(u_1)+\nu(G-N[u_1])\leqs{r},$$ the claim is true.
Now we prove the result according to the following two cases.

Case 1. $\Dta(G)<r$.  Then, by Lemma \ref{LEMMA:f_Nu_Delta},  we have $e(G)\leqs{f(\nu, \Delta)}\le r(r-1+1)=r^2$, and the equality holds only if $\nu=r$ and $\Delta=r-1$. We claim that the equality does not hold in this case. For $r\ge 4$, $e(G)\le f(r, r-1)=r(r-1)+{\lf\frac{r-1}{2}\rf}{\lf\frac{r}{\lc (r-1)/2\rc}\rf}<r^2$. For $r=3$, $e(G)\leqs f(3, 2)=3^2=9$. Suppose to the contrary that $e(G)=9$.  Since $\Delta(G)\le 2$ and $G$ has no isolated vertex, $|V(G)|\ge e(G)=9$ (the equality holds if and only if $G$ is 2-regular) and $\omega(G)\le \nu(G)=3$. By Observation~\ref{OBS: o1}, $$3=\nu(G)\ge \frac{|V(G)|-\omega(G)}2\ge \frac{|V(G)|-3}2.$$ Hence $|V(G)|\ge 9$. Thus $|V(G)|=9$ (and so $G$ is 2-regular) and $\omega(G)=3$.  Therefore, $G$ consists of  three vertex-disjoint triangles. But this contradicts the assumption that  $\deg(x)+\nu(G-N[x])\leqs{r}$ for all $x\in{V(G)}$.

\vspace{5pt}
 Case 2. $\Dta(G)=r$. Choose $x\in{V(G)}$ such that $\deg(x)=r$, then $\nu(G-N[x])=0$. Hence $e(G-N[x])=0$ and so each vertex in $G-N[x]$ must be adjacent to  vertices in $N(x)$. Let $N(x)=\{x_1,\cdots,x_r\}$. For each $i\in [1,r]$, let $d_i=\deg(x_i)$ and $\tilde{d_i}=\deg_{G[N(x)]}(x_i)$. Then
\begin{eqnarray*}
e(G)&=&e(G[N(x)])+e(N(x),V(G)\setminus N(x))=\frac12\sum_{i=1}^r\tilde{d_i}+\sum_{i=1}^r(d_i-\tilde{d_i})\\
&=&\sum_{i=1}^rd_i-\frac12\sum_{i=1}^r\tilde{d_i}\leqs r^2-\frac12\sum_{i=1}^r\tilde{d_i}\leqs r^2.
\end{eqnarray*}
Moreover, the equality holds if and only if $d_i=r$ and $\tilde{d_i}=0$ for each $i\in [1,r]$, that is $G$ is a bipartite graph with partites $N(x)=\{x_1, \cdots, x_r\}$ and $V(G)\setminus N(x)$. To show that $G=K_{r,r}$, it suffices to prove that $|V(G)\setminus N(x)|=r$. If not, then $|V(G)\setminus N(x)|>r$. Since $\deg(x_1)=d_1=r$, there must exist a vertex $y\in (V(G)\setminus N(x))\setminus N(x_1)$. Since $G$ has no isolated vertex, $y$ must be adjacent to some vertex $x_j$ with $j\not= 1$. This implies that  $\nu(G-N[x_1])\ge 1$, a contradiction with $\deg(x_1)+\nu(G-N[x_1])\le r$.
\end{proof}

The following lemma states that the members $\mathcal{F}_{n,k}$ are $C_{k,q}-$free.
\begin{lem}\label{LEMMA:Gnk is Ckq-free}
Each member of $\mathcal{F}_{n,k}$ is $\ckq$-free for all $k\ge 2$, $n\geqs4(k-1)^2$, and odd integer $q\ge 5$.
\end{lem}
\begin{proof}
Suppose to the contrary that there is a graph $G\in \mathcal{F}_{n,k}$ containing a copy of $\ckq$. Let $K$ be the copy of $K_{k-1,k-1}$ in $G$. Then each odd cycle of $C_{k,q}$ must contain odd number of the edges of $K$. Let $A=E(C_{k,q})\cap E(K)$. Then $|A|\ge k$. We claim that the center of $C_{k,q}$ must lie in $K$.   If not, then $G[A]$ contains a matching of order at least $k$ by the structure of $C_{k,q}$, a contradiction with $\nu(K)=k-1$. Let $x\in V(K)$ be the center of $\ckq$.  Assume that $\deg_{G[A]}(x)=s$ and let $E_A(x)$ be the set of edges incident with $x$ in $G[A]$. Then at most $s$ cycles of $C_{k,q}$ intersect  $E_A(x)$, that is $A-E_A(x)$ contains a matching of $K$ of order at least $k-s$. This is impossible since $\nu(K-N_{G[A]}(x))\le k-s-1$.
\end{proof}


\begin{lem}\label{Lemma:deltaGgeq} 
 Let $n_0$ be an integer and let $G$ be a graph on $n\ge n_0+{n_0\choose 2}$ vertices with $e(G)=e(T_{n,2})+j$ for some integer $j>0$. Then $G$ contains a subgraph $G'$ on $n'> n_0$ vertices  such that $\delta(G')\geqs\delta(T_{n',2})$ and $e(G')\geqs e(T_{n',2})+j+n-n'$.
\end{lem}
\begin{proof}
If $\delta(G)\ge \lf\frac{n}2\rf$, then $G$ is the desired graph and we have nothing to do. So assume that $\delta(G)<\lf\frac{n}2\rf$. Choose $v\in V(G)$ with $\deg_G(v)<\dtaT$. Let $G_1= G-v$. Then $e(G_1)\geqs e(G)-\deg_G(v)\geqs e(T_{n,2})+j-\dtaT+1=e(T_{n-1,2})+j+1$, since $e(T_{n,2})- e(T_{n-1,2})=\dtaT$. We may continue this procedure until we get $G'$ on $n-i$ vertices with $\delta(G')\geqs\delta(T_{n-i,2})$ for some $i<n-n_0$, or until $i=n-n_0$. For the latter case, $G'$ has $n_0$ vertices but $e(G')\geqs e(T_{n_0,2})+j+i>n-n_0\geqs\binom{n_0}{2}$, which is impossible.
\end{proof}

\section{Proof of Theorem \ref{THM:MAIN THEOREM}}
Let $G$ be an extremal graph for $\ckq$ on $n\geqs n_1(k,q)$ ($n_1(k,q)$ is given below) vertices. By Lemma \ref{LEMMA:Gnk is Ckq-free}, $e(G)\geqs e(T_{n,2})+(k-1)^2$.
We will show that $e(G)=e(T_{n,2})+(k-1)^2$ and $G$ is a member of $\mathcal{F}_{n,k}$. Let
\begin{eqnarray*}
&&\gm=\frac 1{1600},\\
&&n_0=n_0(\ckq,\gm)\ \  (\mbox{ which is determined by $\ckq$ and $\gm$ by applying Lemma~\ref{LEMMA:partition}}), \\
&&n_1=n_1(k,q)=n_0+20k^2q+\binom{n_0+20k^2q}{2}.
\end{eqnarray*}

By the choice of $n_1$ and Lemma \ref{Lemma:deltaGgeq}, we may assume $\delta(G)\ge \delta(T_{n,2})=\lf\frac{n}2\rf$, otherwise, we consider a subgraph $G'$ with the desired minimum degree instead of $G$.
 Let $V_0\dot{\cup}V_1$ be a partition of $V(G)$ such that $e(V_0,V_1)$ is maximized. Lemma \ref{LEMMA:partition} implies that $m={e(V_0)+e(V_1)}<\gmnn$.
The following claim asserts that the partition is closed to be balanced.

\begin{claim}\label{LEMMA:V0V1 are closely balanced}
$$\frac{n}{2}-\sgmn<\left|V_i\right|<\frac{n}{2}+\sgmn \ \  \mbox{for $i=0,1$.}$$
Furthermore, $m=e(V_0)+e(V_1)\ge (k-1)^2$ and if the equality holds then $G$ contains a complete balanced bipartite graph with classes $V_0$ and $V_1$.
\end{claim}
\begin{proof}
Let $|V_0|=\frac{n}{2}+a$. Then $|V_1|=\frac{n}{2}-a$. Since
$$\lf\frac{n^2}{4}\rf+(k-1)^2=e(T_{n,2})+(k-1)^2\leqs e(G)\leqs|V_0||V_1|+m=\frac{n^2}{4}-a^2+m,$$
we have
$m\geqs(k-1)^2$ and $m\geqs a^2$. Since $m<\gmnn$, $a^2<\gmnn$. Hence $|a|<\sgmn$.

If $m=(k-1)^2$, then $$e(T_{n,2})+(k-1)^2\leqs e(G)=e(V_0,V_1)+(k-1)^2.$$
 Hence $e(V_0,V_1)=e(T_{n,2})$, that is $V_0, V_1$ are balanced and so $G$ contains a complete balanced bipartite graph with classes $V_0$ and $V_1$.
\end{proof}
In the following, let $G_i=G[V_i]$, $\Dta_i=\Dta(G_i)$ and $\nu_i=\nu(G_i)$, $i=0,1$ for short. For a vertex $x\in V_i$,  let $E_{1-i}(x)=\{e\in E(G_{1-i}) |\ V(e)\cap N_{G}(x)\not=\emptyset\}$.

\begin{claim}\label{Claim: C_k,q} For any vertex $x\in V_i$,  $$\deg_{G_i}(x)+\nu(G_i-N_{G_i}[x])+\nu(G[E_{1-i}(x)])\leqs k-1.$$
\end{claim}
\begin{proof}We prove it by contradiction. Wlog, assume that there is an $x\in V_0$ such that $\deg_{G_0}(x)+\nu(G_0-N_0[x])+\nu(G[E_{1}(x)])\ge k$. Let $xx_1, xx_2,\ldots, xx_s\in E(G_0)$, and let $M_0=\{u_{s+1}v_{s+1}, \ldots, u_{t}v_{t}\}$ be a matching of $G[V_0\setminus\{x,x_1,\ldots, x_s\}]$ and $M_{1}=\{u_{t+1}v_{t+1}, \ldots, u_{k}v_{k}\}$ a matching of $G[E_{1}(x)]$ such that $xu_{t+1}, \ldots, xu_{k}\in E_G(x, V_{1})$, where $s,t\in [0,k]$. Define
	$$
	\xi(j)=\left\{
	\begin{aligned}
	&1 & &\mbox{if~$j$~is~odd,}\\
	&0 & &\mbox{if~$j$~is~even.}
	\end{aligned} \right.
	$$
	We say that $v$ is {\it bad} if $\deg_{G[V_i]}(v)>t_1$, otherwise $v$ is said to be {\it good}, where  $t_1=6\sgmn$. Then the number of bad vertices in $G$ is at most $\frac{2m}{t_1}<\frac{2\gmnn}{t_1}=\bad.$
	
For any vertex $u\in V_i$ ($i=0,1$), by the maximality of $e(V_0,V_1)$, we have
	 $$e_G(u,V_{1-i})\geqs \max\{e_G(u,V_i),\dtaT-e_G(u,V_i)\}\geqs\frac12\dtaT.$$
Particularly, if $u\in {V_i}$  is good, then we have
	 $$e_G(u,V_{1-i})\geqs\dtaT-e_G(u,V_i)\geqs\dtaT-t_1~(\geqs\frac12\dtaT).$$
	
	Let $A=V(\{xx_1, xx_2,\ldots, xx_s\}\cup M_i\cup M_{1-i})$. We find a copy of $\ckq$ passing through all the vertices of $A$ to get a contradiction.
	
	For each $\ell\in [1,s]$, we find a sequence of vertices $w_{1\ell}^1, w_{0\ell}^2,\cdots, w_{0\ell}^{q-3}, w_{1\ell}^{q-2}$ with $w_{\xi(j)\ell}^j\in{V_{\xi(j)}\setminus A}$ for $1\leqs j\leqs q-2$, such that $w_{0\ell}^{q-3}$ is good and $xw_{1\ell}^1\cdots w_{0\ell}^{q-3}w_{1\ell}^{q-2}x_{\ell}x$ is a $q$-cycle. Furthermore,  we require that $w_{\xi(j)\ell}^j$ ($\ell\in[1,s]$, $j\in [1,q-2]$) are pairwise different. This is possible since together with all vertices in $A$, the total number of good vertices which we have found is at most $|V(\ckq)|=k(q-1)+1$ and each vertex $u\in V_i$ has at least
	$$e_G(u, V_{1-i})-\bad\geqs\frac12\dtaT-\bad\geqs|V(\ckq)|~(\mbox{since}~n\geqs 20k^2q)$$
	good neighbors in $V_{1-i}$
	and the number of common neighbors of $w_{0\ell}^{q-3}$ and $x_{\ell}$ in $V_{1}$ is at least $(\mbox{since}~n\geqs 20k^2q)$
	$$ e_{G}(w_{0\ell}^{q-3},V_{1})+e_{G}(x_{\ell},V_{1})-|V_{1}|\geqs\dtaT-t_1+\frac12\dtaT-(\frac{n}{2}+\sgmn)\geqs|V(\ckq)|.$$
	Thus we have found a copy of $C_{s,q}$ centered at $x$ and passing through the edges of $\{xx_1, xx_2,\cdots, xx_s\}$. 	Particularly, since $G$ is $\ckq-$free, we have $s\leqs k-1$. Thus $$\Delta_i=\Delta(G_i)\leqs k-1<t_1,\  i=0,1.$$ Consequently, all the vertices of $G$ are good, and for each vertex $u\in V_i$, $$e_G(u, V_{1-i})\geqs\dtaT-(k-1)\geqs\frac{n}{2}-k$$ and thereby $$\frac{n}{2}-k\leqs|V_i|\leqs\frac{n}{2}+k, \ i=0,1.$$
	
	Next we will find a copy of $C_{k-s, q}$ centered at $x$ disjoint from the copy of $C_{s,q}$. For every $u_{\ell}$ ($\ell\in [s+1,t]$), choose a common neighbor of $x$ and $u_l$, say $w_{1\ell}^{1}$, in $V_1$ such that $w_{1\ell}^1\neq w_{1\ell'}^{1}$ if $\ell\neq \ell'$. We can do this because the number of common neighbors of $x$ and $u_\ell$ in $V_1$ is at least
	$$e_G(x, V_1)+e_G(u_\ell, V_1)-|V_1|\geqs 2(\frac n2-k)-(\frac n2+k)\geqs|V(\ckq)|\ (\mbox{since}~n\geqs 20k^2q).$$
	For each $\ell\in[s+1,t]$, with the same reason as above,  we find vertices $w_{1\ell}^3,w_{0\ell}^4,\cdots,w_{0\ell}^{q-3}$ one by one, then a common neighbor of $w_{0\ell}^{q-3}$ and $v_\ell$, say $w_{1\ell}^{q-2}$,  in $V_1$ such that $u_\ell w_{1\ell}^1xw_{1\ell}^3\cdots w_{1\ell}^{q-2}v_\ell u_\ell$ is a $q$-cycle. And for every $l\in [t,k]$, begin with $x$, we  can find vertices $w_{1\ell}^2,w_{0\ell}^3,\cdots,w_{1\ell}^{q-3}$ one by one, then a common neighbor, say $w_{0\ell}^{q-2}$, of $w_{1\ell}^{q-3}$ and $v_\ell$ in $V_0$ such that $u_\ell xw_{1\ell}^2w_{0\ell}^3\cdots w_{0\ell}^{q-2}v_\ell u_\ell$ is a $q$-cycle.
	Also, we may require that $w_{\xi(j)\ell}^j\ (\ell\in[s+1, k], j\in[1,q-2])$ are pairwise different. Thus the $k-s$ $q$-cycles form a copy of $C_{k-s,q}$ centered at $x$ and passing through all of the edges of $M_0\cup M_1$.  Therefore, $G[\{x,x_1,\cdots,x_l\}\cup\{w_{\xi(j),l}^j \ : \ j\in[1,q-2], \ \ell\in [1,k]\}]$ contains a desired copy of $\ckq$ with center $x$. This completes the proof.
\end{proof}

\begin{claim}\label{CLAIM:Nu0+Nu1<k} $\nu_0+\nu_1\leqs{k-1}$.
\end{claim}
\begin{proof} If not, suppose that $M_0=\{u_1v_1,\cdots,u_sv_s\}$ and $M_1=\{u_{s+1}v_{s+1},\cdots,u_kv_k\}$ are matchings in $G_0$ and $G_1$, respectively. Wlog, assume that $s\geqs1$.
First we find a common neighbor, say $x$, of $u_{s+1},\cdots,u_k$ in $V_0\setminus V(M_0)$. This is possible since the number of such neighbors is at least
\begin{eqnarray*}
&&\sum_{l=s+1}^k e_G(u_\ell, V_0)-(k-s-1)|V_0|-2s\\
&\geqs&(k-s)(\frac n2-k)-(k-s-1)(\frac{n}{2}+k)-2k\\
&=&\frac n2-(2(k-s)+1)k\\
&\geqs& \frac n2-2k^2>0 ~(\mbox{since}~n\ge 20k^2q).
\end{eqnarray*}
Let $M_0'$ be the maximal subset of $M_0$ such that $M_0'$ is a matching of $G_0-N_{G_0}[x]$.
By Claim \ref{Claim: C_k,q}, $\deg_{G_0}(x)+|M_0'|+|M_{1}|\le k-1$. Clearly, $\deg_{G_0}(x)+|M_0'|\ge s$. Hence $\deg_{G_0}(x)+|M_0'|+|M_{1}|\ge k$, a contradiction.
\end{proof}

\begin{claim}\label{CLAIM:Dta and Nu attaining max}
$\max\{\Dta_0,\Dta_1\}=k-1$.
\end{claim}
\begin{proof} If not, then by Claim \ref{Claim: C_k,q}, $\max\{\Dta_0,\Dta_1\}\leqs k-2$. Thus by Lemma \ref{LEMMA:f_Nu_Delta} and Claim \ref{CLAIM:Nu0+Nu1<k},
\begin{eqnarray*}
m&=&e(V_0)+e(V_1)\leqs f(\nu_0,k-2)+f(\nu_1,k-2)\\
&\leqs&f(\nu_0+\nu_1,k-2)\leqs f(k-1,k-2).
\end{eqnarray*}

If $k\neq4$, $m\leqs f(k-1,k-2)=(k-1)^2-1$, contradicts to $m\ge (k-1)^2$ (by Claim~\ref{LEMMA:V0V1 are closely balanced}).

If $k=4$, then $m\leqs f(3,2)=(k-1)^2=9$. By Claim~\ref{LEMMA:V0V1 are closely balanced}, $m=(k-1)^2=9$ and $G$ contains a complete balanced bipartite subgraph with classes $V_0$ and $V_1$. Let $H$ be the subgraph consisting of nonempty components of $G_0\cup G_1$.  Then $H$ is a graph with $e(H)=9$, $\Delta(H)=2$ and $\nu(H)=3$. By Observation~\ref{OBS: o1} and with a similar discussion as in  Case 1 of the proof of Lemma~\ref{LEMMA:MAIN LEMMA}, $H$  consists of three vertex-disjoint triangles. Then we can easily find a vertex $x$ in $H$ with $\deg_H(x)=2$ and a matching of order 2 in $H-N_H[x]$. That is $\deg_{G_i}(x)+\nu(G_i-N_{G_i}[x])+\nu(G[E_{1-i}(x)])\ge 4=k$, a contradiction with Claim \ref{Claim: C_k,q}.
\end{proof}

\begin{claim}\label{CLAIM:e0e1=0} $e(V_0)\cdot e(V_1)=0$.
\end{claim}
\begin{proof} At first, by Claim \ref{CLAIM:Dta and Nu attaining max} and Claim~\ref{CLAIM:Nu0+Nu1<k}, we have
	\begin{eqnarray*}
		m&\leqs& f(\nu_0,k-1)+f(\nu_1,k-1)\leqs f(\nu_0+\nu_1,k-1)\\
		&\leqs& f(k-1,k-1)\leqs k(k-1).
	\end{eqnarray*}
	
	Next, by Claim \ref{CLAIM:Dta and Nu attaining max}, $\max\{\Dta_0,\Dta_1\}=k-1$. Wlog, assume $\Dta_0=k-1$. Let $x\in V_0$ with $\deg_{G_0}(x)=k-1$. We show that $e(V_1)=0$.
If $e(V_1)>0$, then $\nu_1\geqs1$. By Claim \ref{CLAIM:Nu0+Nu1<k}, $\nu_0\leqs k-1-\nu_1=k-2$. Let $A_1=\{u\in{V_1}:\deg_{G_1}(u)>0\}$.
By Claim \ref{Claim: C_k,q}, we have $A_1\cap N_G(x)=\emptyset$.
So $e(V_0,V_1)\leqs|V_0||V_1|-|A_1|\leqs e(T_{n,2})-|A_1|$. Thus we have  $$e(T_{n,2})+(k-1)^2\leqs e(G)\leqs e(T_{n,2})-|A_1|+m.$$
Therefore,  $|A_1|\leqs m-(k-1)^2$( $\leqs k-1$).  Again by Lemma~\ref{LEMMA:f_Nu_Delta},  we have
\begin{eqnarray*}
m
 &\leqs &\nu_0(\Dta_0+1)+\nu_1(\Dta_1+1)\\
 &\leqs& k\nu_0+(k-1-\nu_0)|A_1|\, \mbox{ (since $\Dta_1+1\leqs|A_1|$)}\\
&=&\nu_0(k-|A_1|)+(k-1)|A_1|\\
&\leqs&(k-2)(k-|A_1|)+(k-1)|A_1|\, \mbox{ (since $|A_1|\le k-1$ and $\nu_0\le k-2$)}\\
& =&(k-1)^2+|A_1|-1\leqs(k-1)^2+m-(k-1)^2-1\\
&= & m-1, \mbox{ a contradiction.}
\end{eqnarray*}
\end{proof}

By Claim \ref{CLAIM:e0e1=0}, wlog, we may assume $e(V_1)=0$. So $m=e(V_0)$. Let $A_0$ be the set of non-isolated vertices in $G_0$. By Claim \ref{Claim: C_k,q} and Lemma \ref{LEMMA:MAIN LEMMA}, $m=e(G[A_0])\leqs(k-1)^2$. By Claim~\ref{LEMMA:V0V1 are closely balanced}, we must have $m=(k-1)^2$ and therefore $G$ contains a complete balanced bipartite subgraph with classes $V_0$ and $V_1$. Again by Claim \ref{Claim: C_k,q} and Lemma \ref{LEMMA:MAIN LEMMA}, since $e(G[A_0])=(k-1)^2$, $G[A_0]$ must be a copy of $K_{k-1,k-1}$. This completes the proof of Theorem~\ref{THM:MAIN THEOREM}.
\qed

\end{document}